\documentclass[11pt,a4paper]{article}
\usepackage[utf8]{inputenc}
\usepackage[margin=0.8in]{geometry}
\usepackage{amsmath, amssymb, amsthm}
\usepackage{tikz}
\usetikzlibrary{cd, shapes.geometric, arrows.meta, positioning, calc}

\usepackage[colorlinks=true, linkcolor=blue, pagebackref=true]{hyperref} 

\usepackage{enumerate}

\theoremstyle{plain}
\newtheorem{theorem}{Theorem}[section]
\newtheorem*{theorem*}{Theorem}
\newtheorem{proposition}[theorem]{Proposition}
\newtheorem{lemma}[theorem]{Lemma}
\newtheorem{corollary}[theorem]{Corollary}
\newtheorem{remark}[theorem]{Remark}

\newtheorem{exam}[theorem]{Example}
\theoremstyle{definition}
\newtheorem{definition}[theorem]{Definition}
\newtheorem{assum}[theorem]{Assumption}
\newtheorem*{notation}{Notation}

\newcommand{\coker}{\mathrm{coker}}
\newcommand{\im}{\mathrm{Im}}

\newcommand{\CO}{\mathcal{O}}
\newcommand{\spa}{\mathrm{Spa}}
\newcommand{\Spa}{\mathrm{Spa}}
\newcommand{\Spec}{\mathrm{Spec}}
\newcommand{\Spf}{\mathrm{Spf}}

\newcommand{\Hom}{\mathrm{Hom}}
\newcommand{\Tor}{\mathrm{Tor}}

\newcommand{\CV}{\mathcal{V}}
\newcommand{\Ext}{\mathrm{Ext}}

\begin{document}
\title{Almost Vector Bundles over Perfectoid Spaces}
\author{Yuntong Cui, \ Guo Li, \  Shuhan Jiang\ and Jiahong Yu\thanks{Academy of Mathematics and Systems Science, Chinese Academy of Sciences} \thanks{Morningside Center of Mathematics, CAS}}
\maketitle

\begin{abstract}
   In this paper, we define vector bundles within the framework of almost mathematics (referred to as \emph{almost vector bundles}) and establish the $v$-descent theorem together with a structure theorem for these bundles over perfectoid spaces. The proof yields several interesting intermediate results.
\end{abstract}

\tableofcontents
\section{Introduction}
In \cite{scholze2012perfectoid}, Scholze defined integral perfectoid rings as an analogue of perfect rings in mixed characteristic. Integral perfectoid rings play an important role in the modern theory of commutative algebra, algebraic geometry and $p$-adic Hodge theory.

For a perfect ring $A$, the following theorem was proved in \cite{bhatt2017projectivity}. 
\begin{theorem*}[{\cite[Theorem 4.1]{bhatt2017projectivity}}]
    The restriction functor induces an equivalence between the category of $v$-vector bundles on $\Spec(A)$ and the category of Zariski vector bundles on $\Spec(A)$.
\end{theorem*}
Using this theorem, Bhatt and Scholze further proved the Witt vector Grassmannian is represented by a scheme (not merely as an algebraic space).

As an analogue of perfect rings, we expect that integral perfectoid rings also admit similar descent theorems. 
Specifically, for a perfectoid Tate--Huber pair \((A,A^+)\), let \(X = \mathrm{Spa}(A,A^+)\). We wish to establish the relationship between vector bundles with coefficients in \(\mathcal{O}_X^+\) on \(X\) under the \(v\)-topology and vector bundles with coefficients in \(\mathcal{O}_X^+\) on \(X\) under the analytic topology. However, the article \cite{heuer2024gtorsorsperfectoidspaces} (see also Example \ref{ex: et bundle not descent}) has pointed out that even vector bundles with coefficients in \(\mathcal{O}_X^+\) on \(X\) under the \'{e}tale topology may not necessarily descend to analytic vector bundles. In this paper, we will use the language of almost mathematics to bridge this gap. In particular, we will define vector bundles in the context of almost mathematics (called almost vector bundles) and show that all almost vector bundles with coefficients in \(\mathcal{O}_X^+\) on \(X\) under the \(v\)-topology can descend to almost vector bundles with coefficients in \(\mathcal{O}_X^+\) on \(X\) under the analytic topology.

\subsection{Main results}\label{subsec:main}

Let $(A,A^+)$ be a perfectoid Tate--Huber pair and $X=\mathrm{Spa}(A,A^+)$. To make the exposition more concise, we introduce the following notation:
\begin{notation}
	Let $A^{\circ\circ}\subseteq A^+$ be the ideal of all topologically nilpotent elements in $A$, and consider the almost mathematics relative to the ideal $A^{\circ\circ}$. Denote by \( X_{\mathrm{an}} \) (resp. \( X_v \)) the analytic (resp. $v$-) topology on \( X \). Define \( X_{\mathrm{an}}^{+,a} \) (resp. \( X_v^{+,a} \)) as the almost ringed site \( (X_{\mathrm{an}}, \mathcal{O}^{+,a}_X) \) (resp. \( (X_v, \mathcal{O}^{+,a}_X) \)) in sense of Definition \ref{defi: almost ringed site setting}.
\end{notation}

Using this notation, we define almost vector bundles as follows:
\begin{definition}[Definition \ref{def}]\label{defi: almost vb intro}
	Let $?\in\{\mathrm{an},v\}$. For a positive integer $N$, a sheaf of module $\mathcal{V}$ on $X_{?}^{+,a}$ is called an \emph{almost vector bundle of rank $N$} if for any $\alpha\in R^{\circ\circ}$, there exists a covering $Y\to X$ of $X$ under the $?$-topology and an $\mathcal{O}_Y^{+,a}$-module sheaf homomorphism $f:\big(\mathcal{O}_{Y}^{+,a}\big)^{\oplus N}\to \mathcal{V}|_{Y}$ such that both $\ker f$ and $\mathrm{coker} f$ are annihilated by $\alpha$.
\end{definition}

Introducing almost mathematics is natural. Indeed, the ringed space \( (X_{\mathrm{an}}, \mathcal{O}_X^+) \) lacks properties that an affine space should possess, such as the acyclicity of the structure sheaf. However, it was already shown in \cite{scholze2012perfectoid} that the ringed space \( X_{\mathrm{an}}^{+,a} \) is closer to being an affine space.

The first main result of this paper is to prove the $v$-descent of almost vector bundles. Specifically, we have the following theorem.
\begin{theorem}[Theorem \ref{SAAVB}]\label{theo: v descent of almost vb}
    The restriction functor defines an equivalence between the category of almost vector bundles on \( X_v^{+,a} \) and the category of almost vector bundles on \( X_{\mathrm{an}}^{+,a} \).
\end{theorem}

Next, we study the local structure of almost vector bundles. In the classical (i.e., non-almost) case, this is trivial because by the definition, a vector bundle of rank $N$ is locally isomorphic to the direct sum of $N$ copies of the structure sheaf. However, as can be seen from Definition \ref{defi: almost vb intro}, this is not trivial in almost mathematics. For example, when $X$ is a perfectoid point, by definition, an almost vector bundle on $X_{\mathrm{an}}^{+,a}$ is equivalent to an almost free module of rank $N$ over $A^+$, but an almost free module is not necessarily isomorphic to the direct sum of $N$ copies of $A^{+,a}$. However, we find that under the $v$-topology, the following theorem holds:

\begin{theorem}[Theorem \ref{Thm: SC over v}]\label{theo: locally free of almost vector bundle}
	Let $\mathcal{V}$ be an almost vector bundle of rank $N$ on $X_v^{+,a}$. Then there exists a $v$-covering $f: Y \to X$ of $X$ such that \[\mathcal{V}|_Y \cong \left( \mathcal{O}_Y^{+,a} \right)^{\oplus N}.\]
\end{theorem}

We note that even in the case of a perfectoid point, Theorem \ref{theo: locally free of almost vector bundle} is still non-trivial. Moreover, the proof of Theorem \ref{theo: locally free of almost vector bundle} is a globalization the case of a perfectoid point. To prove the case of a perfectoid point, we make the following interesting observation:

\begin{theorem}[Theorem \ref{perfectoid}]
	Let $K$ be a spherically complete non-Archimedean field with residue field of characteristic $p$. If $K$ additionally satisfies:
\begin{enumerate}[(\roman*)]
	\item The value group of $K$ is $p$-divisible;
	\item The residue field of $K$ is perfect,
\end{enumerate}
then $K$ is a perfectoid field.
\end{theorem}

\begin{remark}
    The assumption that $K$ is spherically complete is necessary. For a counterexample, consider the complete residue field of a type $4$ point on a Berkovich curve over $\mathbb{C}_p$.
\end{remark}

As a corollary, we have:

\begin{corollary}[Lemma \ref{Lem: extension}]\label{cor: embed to sph cpt perfectoid}
	Let $K$ be a non-Archimedean field with residue field of characteristic $p$. Then there exists a non-Archimedean field extension $\widetilde{K}$ of $K$ such that $\widetilde{K}$ is spherically complete and a perfectoid field. Moreover, the value group of $\widetilde{K}$ can be selected to be $\mathbb{R}_{>0}$.
\end{corollary}

Corollary \ref{cor: embed to sph cpt perfectoid} is the core of the proof of Theorem \ref{theo: locally free of almost vector bundle}, and we will discuss these results in detail in Section \ref{Sec: Spherical completeness}.

Furthermore, we observe that almost vector bundles essentially reflect the behavior of locally free \( \mathcal{O}_X^+ \)-modules (honest, not almost) under the $arc$-topology of $\Spf(A^+)$. The $arc$-topology was introduced in \cite{BhattMathew21} (for schemes) and \cite{Bhatt_2022} (for formal schemes) as a refinement of the traditional $v$-topology; unlike the $v$-topology, the $arc$-topology only detects the Berkovich spectrum (rather than the full adic spectrum). We further prove the following theorem.
\begin{theorem}[Section \ref{section:prove arc=v}]\label{theo: avb=arc vb}
    Adopt the notation of Theorem \ref{theo: v descent of almost vb}. Let \( \mathrm{Perf}_{A^+}^{\mathrm{tf}} \) denote the category of \( A^{\circ\circ} \)-torsion-free affinoid integral perfectoid formal schemes over \( A^+ \), endowed with the $arc$-topology. Let \( \mathcal{O}_{\mathrm{Perf}_{A^+}^{\mathrm{tf}}} \) be its structure sheaf (cf. Definition \ref{def: arc sheaf}). Then there is an equivalence between the category of almost vector bundles on \( X_{\mathrm{an}}^{+,a} \) and the category of locally free \( \mathcal{O}_{\mathrm{Perf}_{A^+}^{\mathrm{tf}}} \)-modules on \( \mathrm{Perf}_{A^+}^{\mathrm{tf}} \).
\end{theorem} 

\subsection{Link with previous works}

The $v$-descent of vector bundles is not a new topic in perfectoid theory. We briefly summarize the previous works. Adopt the notation of Subsection \ref{subsec:main}. The work \cite{kedlaya2019relative} established that vector bundles on \(X = \Spa(A, A^+)\) possess the property of \(v\)-descent.
\begin{theorem}[{\cite[Theorem 3.5.8]{kedlaya2019relative}}]\label{theo: rational v descent}
	The following categories are equivalent:
	\begin{enumerate}
		\item $\mathcal{O}_X$-vector bundles  for the  v-topology;
         \item $\mathcal{O}_X$-vector bundles for the \'{e}tale topology;
		\item $\mathcal{O}_X$-vector bundles  for the analytic topology;
		\item Finite-rank projective \(A\)-modules.
	\end{enumerate}
\end{theorem}
This theorem concerns locally free modules with \(\mathcal{O}\)-coefficients (i.e., rational coefficients), not with \(\mathcal{O}_X^+\)-coefficients (i.e., integral coefficients). For \(\mathcal{O}_X^+\)-coefficients, Heuer proved the following theorem.
\begin{theorem}[{\cite[Theorem 2.21]{heuer2024gtorsorsperfectoidspaces}}]\label{theo: v descent of O^+ bundles}
	The following categories are equivalent:
	\begin{enumerate}
		\item Locally free \(\mathcal{O}_X^+\)-modules on \(X_v\);
		\item Locally free \(\mathcal{O}_X^+\)-modules on \(X_{\text{\'et}}\).
	\end{enumerate}
\end{theorem}
However, as mentioned earlier, locally free \(\mathcal{O}_X^+\)-modules do not possess \'{e}tale descent. The primary reason for this phenomenon is that almost purity implies that for a morphism \(f: R \to S\) of integral perfectoid rings, if the localization \(f[\frac{1}{p}]\) is finite étale, then \(f\) itself is only almost finite étale, not necessarily finite étale. This is precisely the motivation for introducing almost mathematics in our consideration.

More recently, Bhatt--Lurie defined the Riemann--Hilbert functor for arbitrary $p$-adic formal schemes. In the process, they defined overconvergent sheaves on $X = \Spa(A, A^+)$ and proved the following theorem:
\begin{theorem}[{\cite[Theorem 5.2.2 \& Remark 5.2.4]{Bhatt2025padichodge}}]
Let $\varpi$ be a pseudo-uniformizer of $A$, and let $\widehat{\mathrm{Shv}}(X, \mathcal{O}_X^{+,a})^{oc}$ be the derived category of $\varpi$-adically complete, overconvergent almost $\mathcal{O}_X^{+,a}$-sheaves. Let $\widehat{D}(A^{+,a})$ be the derived category of $\varpi$-complete almost modules on $A^+$. Then the functor $R\Gamma(X, -)$ defines an equivalence between $\widehat{\mathrm{Shv}}(X, \mathcal{O}_X^{+,a})^{oc}$ and $\widehat{D}(A^{+,a})$.
\end{theorem}
Under our definition, almost vector bundles are overconvergent. By this theorem, we can embed almost vector bundles into $\widehat{D}(A^{+,a})$. However, there is no evidence that the essential image lies in the degree-zero part.

Moreover, sheaves valued in the category of almost modules are also mentioned in the article \cite{Zavyalov_2025}. In \cite{Zavyalov_2025}, the authors point out that for a rigid space $X/\mathbb{C}_p$, the derived pushforward of a vector bundle with coefficients in $\widehat{\mathcal{O}}_X^+$ under the $v$-topology to an integral model of $X$ exhibits coherence in the sense of almost mathematics. In fact, the techniques of \cite{Zavyalov_2025} still work under our definition, hence most results of \cite{Zavyalov_2025} can be generalized to almost vector bundles.

\subsection{Acknowledgments}

The idea for this paper originated during the 2025 ``Algebra and Number Theory" summer school. The authors would like to express their sincere gratitude to the Institute of Mathematics and Systems Science of the Chinese Academy of Sciences, the School of Mathematical Sciences at the University of Chinese Academy of Sciences, and the School of Mathematical Sciences at Peking University for organizing the event, as well as to the professors for their insightful lectures.

\section*{Notations and Conventions}
\begin{itemize}
    \item All Tate rings are assumed to be complete.
    \item For any adic space $X$ and $x\in X$, denote by $k(x)$ (resp. $k(x)^+$) the \textbf{complete} residue field of $x$ (resp. the valuation ring of $k(x)$).
    \item For any ring $R$, an element $a\in R$ and an $R$-module $M$, denote by $M^{a-\text{tf}}$ the maximal $a$-torsion free quotient of $M$.
\end{itemize}

\section{Preliminaries on Adic Spaces}

Fix a sheafy Tate--Huber pair $(R,R^+)$. Let $X=\Spa(R,R^+)$. Fix a pseudo-uniformizer $\varpi\in R^+$. Fix a ring of definitions $R_0\subseteq R^+$ of $R$.

\begin{theorem}\label{theo: bi cartesion of of stalks}
For $x\in X$, let $\mathcal{O}_{X,x}$ (resp. $\mathcal{O}_{X,x}^+$) be the stalk of the sheaf $\mathcal{O}_{X}$ (resp. $\mathcal{O}_{X}^+$) at $x$. Then the diagram
\[
\begin{tikzcd}
\mathcal{O}_{X,x}^+ \arrow[r] \arrow[d] & k(x)^+ \arrow[d] \\
\mathcal{O}_{X,x} \arrow[r] & k(x)
\end{tikzcd}
\]
is bi-cartesian.
\end{theorem}

\begin{proof}
    First, we show that this is a pullback diagram. Let $\Pi:\mathcal{O}_{X,x}\to k(x)$ be the natural projection. For $f\in\mathcal{O}_{X,x}$ such that $\Pi(f)\in k(x)^+$, we need to show that $f\in \mathcal{O}_{X,x}^+$. Suppose $x\in U$ is an open neighborhood of $x$ in $X$ such that $f\in \Gamma(U,\mathcal{O}_X)$, then $\Pi(f)\in k(x)^+$ implies that $|f|_x\leq 1$, where $|-|_x$ is the valuation corresponding to $x$. Take $U'\subseteq U$ to be the rational open subset defined by $|f|\leq 1$, then $U'$ is also an open neighborhood of $x$ and by assumption, we have $f\in\Gamma(U',\mathcal{O}_X^+)$. Therefore, $f\in \mathcal{O}_{X,x}^+$.

    Next, we show that $k(x)\cong \mathcal{O}_{X,x}\otimes_{\mathcal{O}_{X,x}^+}k(x)^+$. Indeed, by definition, $\mathcal{O}_{X,x}=\mathcal{O}_{X,x}^+[\frac{1}{\varpi}]$, so the conclusion follows immediately.
\end{proof}

The following lemma is well-known for experts.

\begin{lemma}\label{lemma: calO_x/varpi=k^+(x)/varpi}
    The canoical projection induces an isomorphism $\mathcal O_{X,x}^+/\varpi\cong k(x)^+/\varpi$.
\end{lemma}

\begin{proof}
    The surjectivity is a direct corollary of the construction of the complete residue fields. For the injectivity, let $f\in \ker(\mathcal O_{X,x}\to k^+(x)/\varpi)$. Choose an open subset $x\in U\subseteq X$ such that $f\in \Gamma (U,\mathcal O_{X}^+)$ and denote by $|-|_x$ the valuation o $x$. Then $|f|_x\leq |\varpi|_x$ by definition. Define $U'\subseteq U$ be the open neighborhood of $x$ defined by $|f|\leq |\varpi|$. Then $f\in \varpi \Gamma(U',\mathcal O_{X}^+)$ by the construction. Thus, $f\in \varpi\mathcal{O}_{X,x}^+$. 
\end{proof}

Fix a finite rank projective module $P$ of $R$. It is defined in \cite[Definition 1.1.11]{kedlaya2017} that there is a canonical topology on $P$ making $P$ complete.

\begin{lemma}
    Any $R^+$-submodule $M\subseteq P$ satisfying that $M[\frac{1}{\varpi}]=P$ is open.
\end{lemma}

\begin{proof}
    Take a family of elements $\{x_1,x_2,\dots,x_n\}\subseteq P$ such that they generate $P$. By multiplying by some power of $\varpi$, they can be made to lie in $M$. This defines an $R$-linear surjection $f:R^n\to P$. By the definition of the canonical topology, $f$ is continuous. Therefore, by the open mapping theorem, $f\big((R^+)^{\oplus n}\big)\subseteq P$ is an open subset. This shows that $M$ contains an open subgroup (under addition), hence is itself open.
\end{proof}

\begin{lemma}\label{lemma: preimage of int module}
    Let $P^+$ be an open $R^+$-submodule of $P$, and denote by $P^+_x$ the image of $\mathcal{O}_{X,x}^+\otimes_{R^+}P^+$ in $\mathcal O_{X,x}\otimes_{R} P$, $\overline P^+_x$ the image of $k(x)^+\otimes_{R^+}P^+$ in $ k(x)\otimes_{R} P$. Let $\Pi_x$ natural projection from $\mathcal{O}_{X,x}\otimes_RP$ to $k(x)\otimes_{R}P$. Then $\Pi^{-1}(\overline{P}_x^+)=P_x^+$.
\end{lemma}

\begin{proof}
    In fact, it suffices to show that the sequence
  \[P^+\otimes_{R^+}\mathcal{O}_{X,x}^+\to \big(\mathcal{O}_{X,x}\otimes_{R}P\big)\oplus \big(k(x)^+\otimes_{R^+}P^+\big)\xrightarrow{m-n} k(x)\otimes_{R}P\to 0\]
  is exact, where the middle arrow maps $(m,n)$ to the difference of their images in $k(x)\otimes_{R}P$. Recall from Theorem \ref{theo: bi cartesion of of stalks} that the diagram
 \[\begin{tikzcd}
\mathcal{O}_{X,x}^+ \arrow[r] \arrow[d] & k(x)^+ \arrow[d] \\
\mathcal{O}_{X,x} \arrow[r] & k(x)
\end{tikzcd}\]
  is a pull-back diagram of rings, so there is an exact sequence of $R^+$-modules
  \[0\to \mathcal{O}_{X,x}^+\to \mathcal{O}_{X,x}\oplus k(x)^+\to k(x)\xrightarrow{x-y} 0.\]
  Using the right-exactness of tensor products immediately gives the conclusion.
\end{proof}

\begin{lemma}\label{lemma: free over int stalk}
    Keep the notation of Lemma \ref{lemma: preimage of int module} and fix a $u\in P$. Let $U_u$ be the set of points $x\in X$ such that the projection of $u$ in $k(x)\otimes_{R}P$ lies in $\overline{P}^+_x$. Then $U_u$ is an open subset.
\end{lemma}

\begin{proof}
    By Lemma \ref{lemma: preimage of int module}, for any point $x\in X$, the image of $u$ in $k(x)\otimes_{R}P$ lying in $\overline{P}_{x}^+$ is equivalent to the image of $u$ in $\mathcal{O}_{X,x}\otimes_{R}P$ lying in $P^+_{x}$. From the property of filtered colimits, we know that if $x\in U_u$, then a neighborhood of $x$ is contained in $U_u$.
\end{proof}

\begin{lemma}
    Let $x\in X$ and $M$ be a $\varpi$-torsion-free finitely generated $\mathcal O_{X,x}^+$-module such that $M[\frac{1}{\varpi}]$ is flat. Then $M$ is free.
\end{lemma}

\begin{proof}
    Recall from \cite[Proposition 1.6]{huber_general} that both $\mathcal O_{X,x}$ and $\mathcal O_{X,x}^+$ are $\varpi$-torsion free local rings. Hence, by \cite[\href{https://stacks.math.columbia.edu/tag/00NZ}{Tag 00NZ}]{stacks-project}, $M[\frac{1}{\varpi}]$ is free over $\mathcal O_{X,x}$ and it suffices to prove that $M$ is flat over $\mathcal O_{X,x}^+$. As $\varpi$ is a nonzero divisor of $\mathcal O_{X,x}^+$, the derived $\varpi$-completion of $M$ (denoted by $\widehat M$) is equal to the classical $\varpi$-adic completion. As $M$ is $\varpi$-torsion free, so is $\widehat{M}$. Hence, \[\widehat{M}/\varpi\cong \mathcal (O_{X,x}^+/\varpi)\otimes_{\mathcal{O}_{X,x}^+}M\cong \big(k^+(x)/\varpi\big)\otimes_{\mathcal O_{X,x}^+}M\]
    is finitely generated. Note that $k(x)^+$ is a valuation ring. By an approximation, $\widehat{M}$ is free over $k(x)^+$. Hence, $(O_{X,x}^+/\varpi)\otimes_{\mathcal{O}_{X,x}^+}M$ is free over $\mathcal O_{X,x}^+/\varpi$. The claim follows from the slicing criterion for flatness (cf. \cite[\href{https://stacks.math.columbia.edu/tag/0H7N}{Tag 0H7N}]{stacks-project}).
\end{proof}

\begin{lemma}\label{free}
    Let $P$ be a finite rank projective $R$-module and $P^+\subseteq P$ be an open $R^+$-submodule. Then for any finite subset $S$ of $P^+$, There exist an analytic cover $Y=\spa(A,A^+)\to X$ and a free $A^+$ module $F\subseteq (A^+\otimes_{R^+}P^+)^{\varpi-\text{tf}}$ containing the images of all elements in $S$.
\end{lemma}

\begin{proof}
    By adding finitely many elements, we may assume that $S$ (via $R$-linear combinations) generates $P$, so it suffices to prove for $P^+=\sum_{s\in S}R^+s$. For any $x\in X$, define $P^+_x$ as in Lemma \ref{lemma: preimage of int module}. Then by Lemma \ref{lemma: free over int stalk}, $P_x^+$ is a free module. Assume $\mathfrak{m}_x\subset \mathcal{O}_{X,x}^+$ is the maximal ideal, take a subset $S'$ of $S$ such that it forms a basis in the $\mathcal{O}_{X,x}^+/\mathfrak{m}_x$-linear space $P_{x}^+/\mathfrak{m}_x$. By Nakayama's lemma, we know that $S'$ is also a basis of $P^+_x$. Define $F_{S'}=\bigoplus_{s\in S'}R^+$ and define an $R^+$-linear map $\rho: \bigoplus_{s\in S'}R^+\to P^+$ sending the basis corresponding to $s\in S'$ to $s\in P^+$. By construction, $\rho$ becomes an isomorphism after base change to $\mathcal{O}_{X,x}^+$, and by the property of filtered colimits, there exists a rational open neighborhood $U=\mathrm{Spa}(R_1,R_1^+)$ of $x$ such that $\rho$ becomes a surjection after base change to $R_1^+$. Since $P$ is a finite-rank projective module, we may shrink $U$ so that $P_1:=R_1\otimes_RP$ is a free module. Comparing the ranks after base change to $\mathcal{O}_{X,x}^+$, we know that the rank of $P_1$ is $|S'|$, which is the same as the rank of $F_{S'}$. Hence, $\rho[\frac{1}{\varpi}]$ becomes an isomorphism after base change to $R_1$. In summary, $\rho$ becomes an isomorphism after base change to $R_1^+$, i.e., $P^+\otimes_{R^+}R^+_1$ is finite free.Performing the above construction for all $x\in X$ then yields the desired result.
\end{proof}

\section{Almost Mathematics and Almost Vector Bundles}

In this section, we review the basic theory of almost mathematics and define almost vector bundles. 

\subsection{Foundations of almost mathematics}

We briefly recall some basic almost mathematics. The main reference is \cite{Gabber2002AlmostRT}.

\begin{assum}[Basic setting of almost mathematics]\label{assum: basic setting of almost mathematics}
    During this subsection, let \( A \) be a ring and \( I \) be an ideal such that \( I \) is flat as an \( A \)-module and $I^2=I$. 
\end{assum}

\begin{lemma}
    Keep the notation in Assumption \ref{assum: basic setting of almost mathematics}. Then, 
    the multiplication map
    \[I\otimes_AI\to I\]
    is an isomorphism.
\end{lemma}

\begin{proof}
    Consider the exact sequence
    \[0\to I\to A\to A/I\to 0.\]
    As $I$ is flat, tensoring the above sequence with $I$ preserves exactness. This means
    \[0\to I\otimes_AI\to I\to I/I^2=0\]
    is exact, and the lemma follows.
\end{proof}

Let \( \mathbf{Mod}_A \) be the category of \( A \)-modules.

\begin{definition}
    Let \( M \in \mathbf{Mod}_A \). The module \( M \) is called \textit{almost zero} if
    \[ I \otimes_A M = 0. \]
\end{definition}

\begin{lemma}[{\cite[Remark 2.1.4]{Gabber2002AlmostRT}}]\label{lemma: IM=0_eq_to_I_tensor_M=0}
    Let $M$ be an $A$-module. Then, $M$ is almost zero if and only if $IM=0$.
\end{lemma}

\begin{proof}
    The `only if' part is obvious. We prove that if $IM=0$, then $I\otimes_AM=0$. Indeed, this time, $M$ is an $A/I$-module. The claim then follows from the following calculation
    \[I\otimes_AM=I\otimes_{A}(A/I\otimes_{A/I}M)=(I\otimes_AA/I)\otimes_{A/I}M=I/I^2\otimes_{A/I}M=0.\]
    The last equality holds by the assumption that $I=I^2$.
\end{proof}

Let \( \mathbf{Mod}_A[I] \) be the full subcategory of \( \mathbf{Mod}_A \) generated by all almost zero modules. Clearly, \( \mathbf{Mod}_A[I] \) is a Serre subcategory \cite[\href{https://stacks.math.columbia.edu/tag/02MN}{Tag 02MN}]{stacks-project}.

\begin{definition}[{\cite[2.2.2]{Gabber2002AlmostRT}}]
    Define the category of \((A,I)\)-almost modules as the quotient category
    \[ \mathbf{Mod}^a_{(A,I)} := \mathbf{Mod}_A / \mathbf{Mod}_A[I]. \]
    Objects in $\mathbf{Mod}_A^a$ are call \emph{almost $A$-modules} or \emph{$A^a$-modules}.
\end{definition}

In practical use, the ideal \( I \) is generally pre-specified, so we simply denote \( \mathbf{Mod}_{(A,I)}^a \) as \( \mathbf{Mod}_A^a \).
The usual tensor product descends to a tensor product on $\mathbf{Mod}_A^a$, which makes $\mathbf{Mod}_A^a$ a monoidal category. Thus, the almost flatness of an almost module is well defined. We will use $-\otimes_{A^a}-$ to denote the tensor product of $A^a$-modules.
Furthermore, an \textit{almost algebra} is defined as a ring object in the almost module category \( \mathbf{Mod}_A^a \).

It is worth mentioning that injective objects behave well in the almost category  \cite[Corollary 2.2.17 \& Corollary 2.2.22]{Gabber2002AlmostRT}. Hence, the right derived functors are well-defined. The left derived functors are more complex. However, it is easy to check that if $\mathcal{F}$ is a right exact functor on $\mathrm {Mod}_A$ such that the left derived functor $L\mathcal{F}$ is zero on $\mathrm {Mod}_A[I]$, then the almostification of $L\mathcal F$ is the left derived functor of the almostification of $\mathcal F$. In particular, the derived tensor product is well defined on $\mathbf{Mod}_A^a$. We will use $\mathrm{Tor}_k^{A^a}$ to denote the $k$-th homology groups of derived tensor products.

\begin{exam}[Almost mathematics of integral perfectoid rings]\label{EX: IPR}
    Let $R$ be an integral perfectoid ring in the sense of \cite[Definition 17.5.1]{Scholze_2020} and let $R^{\circ\circ}$ be the ideal of topologically nilpotent elements. Then $A=R$ and $I=R^{\circ\circ}$ satisfy Assumption \ref{assum: basic setting of almost mathematics}. Indeed, put $R^\flat$ the tilt of $R$ and $\varpi\in R^\flat$ a pseudo-uniformizer. Let $\pi_n\in R$ be the untilt of $\varpi^{1/p^n}$. Then $R^{\circ\circ}=\bigcup_{n\geq 0}\pi_nR$. This shows that $R^{\circ\circ}$ is a colimit of free $R$-module and hence, is flat.
\end{exam}

\begin{exam}[Almost math of non-archimedean fields]
    Let $K$ be a non-archimedean field with norm $|\cdot|:K^\times\to \mathbb{R}_{>0}$. Let $\mathcal O$ be the elements in $K$ with norm non-larger than $1$ and $\frak m$ be the elements of norm smaller than $1$. Assume that $|K^\times|\subseteq \mathbb R_{>0}$ is dense, then $A=\mathcal O$ and $I=\frak m$ satisfy Assumption \ref{assum: basic setting of almost mathematics}.
 \end{exam}

The natural projection from the category $\mathbf{Mod}_A$ to the category $\mathbf{Mod}^a_A$ is called \emph{almostification}. For an $A$-module $M$, denote by $M^a$ the almostification of $M$. By definition, the almostification is exact. One of the most important properties of the almostification functor is that it has left and right adjoints.

\begin{definition}\label{def: * and !}
	Suppose $N\in \mathbf{Mod}_A^a$. Define its module of almost elements as \[N_*=\mathrm{Hom}(A^a,N).\] Via the action of $A$ on $A^a$, regard $N_*$ as an $A$-module. Define
\begin{equation*} N_!:=I\otimes_A  N_*. \end{equation*}
\end{definition}

\begin{theorem}\label{almost}
	The almostification functor $M\mapsto M^a$ has a left adjoint $(N\mapsto N_!)$ and a right adjoint $(N\mapsto N_*)$. Moreover, the left and right adjoints of the almostification functor are both left exact and fully faithful, and the functor $(N\mapsto N_!)$ is also right exact.
\end{theorem}

\begin{proof}
	See \cite[Proposition 2.2.13 \& Proposition 2.2.21]{Gabber2002AlmostRT}.
\end{proof}

Moreover, the essential image of the functor $(N\mapsto N_!)$ can be described by the following lemma.

\begin{lemma}\label{embed}
	The essential image of the functor $(N\mapsto N_!)$ is the modules in $\mathbf{Mod}_A$ for which the product homomorphism $I\otimes_AM\to M$ is an isomorphism.
\end{lemma}

\begin{proof}
	Since the product $I\otimes_AI\to I$ is an isomorphism, for any $N\in \mathbf{Mod}_A^a$, \[I\otimes_AN_!= I\otimes_AI\otimes_AN_*=I\otimes_{A}N_*=N_!.\]

	It remains to show that for any $M\in\mathbf{Mod}_A$ satisfying $I\otimes_AM\cong M$, we have $M\cong M_{!}^a$. In fact, applying the functor $\mathrm{Hom}_A(-,M)$ to the exact sequence
	\begin{equation}\label{eq: proof mod realization of almost mod}0\to I\to A\to A/I\to 0\end{equation}
	yields the exact sequence
	\[0\to \mathrm{Hom}_A(A/I,M)\to M=\mathrm{Hom}_A(A,M)\to \mathrm{Hom}_A(I,M)\to \mathrm{Ext}^1_A(A/I,M).\]
	Here, the second and fifth terms are objects in the category $\mathbf{Mod}_{A}[I]$, so by Lemma \ref{lemma: IM=0_eq_to_I_tensor_M=0}, their tensor product with $I$ vanishes. Since $I$ is flat, applying the exact functor $-\otimes_A I$ to the exact sequence (\ref{eq: proof mod realization of almost mod}) yields
	\[M\cong I\otimes_A\mathrm{Hom}_A(I,M).\]

    Finally, we show that the canonical projection from $\mathrm{Hom}_A(I, M)$ to $\mathrm{Hom}(I^a, M^a)$ is an isomorphism. Verifying injectivity is straightforward; below we verify surjectivity. This can be verified using the concrete construction of the quotient category. In fact, let $\mathrm{Isom}_A^a$ be the class of morphisms $f: M_1 \to M_2$ in $\mathbf{Mod}_A$ such that $\ker f, \coker f \in \mathbf{Mod}_A[I]$. Then, by the definition of the Serre quotient category, $\mathbf{Mod}_A^a$ equals the localized category $\mathbf{Mod}_A[(\mathrm{Isom}_A^a)^{-1}]$. Therefore, any almost module homomorphism $I^a \to M^a$ can be represented by a roof
\begin{equation}\label{eq: roof} I \xleftarrow{p} T \xrightarrow{q} M,\end{equation}
where $p \in \mathrm{Isom}_A^a$. Applying the functor $I \otimes_A -$ to the roof (\ref{eq: roof}) yields the roof
\begin{equation}\label{eq: roof otimes I} I \otimes_A I \xleftarrow{p \otimes_A \mathrm{id}_I} T \otimes_A I \xrightarrow{q \otimes_A \mathrm{id}_I} M \otimes_A I.\end{equation}
By Lemma \ref{lemma: IM=0_eq_to_I_tensor_M=0}, $p \otimes_A \mathrm{id}_I$ is an isomorphism.
Now, noting that the two product homomorphisms $I \otimes_A I \to I$ and $I \otimes_A M \to M$ are both isomorphisms, it follows that the roof (\ref{eq: roof}) is equivalent to the morphism \[(q \otimes_A \mathrm{id}_I) \circ (p \otimes_A \mathrm{id}_I)^{-1}: I = I \otimes_A I \to M \otimes_A I = M.\] In particular, the canonical projection from $\mathrm{Hom}_A(I, M)$ to $\mathrm{Hom}(I^a, M^a)$ is surjective.
\end{proof}

Finally, as the main object of study in this article, we introduce two analogues of freeness in almost mathematics.
\begin{definition}
	Let \( M \in \mathbf{Mod}_A^a \), then
	\begin{enumerate}[(\roman*)]
		\item \( M \) is called \textit{almost finite free of rank $r$} if for every finitely generated ideal \( I_0 \subseteq I \), there exists a homomorphism of $A^a$-modules \( (A^a)^r \subseteq M \) whose kernel and cokernel are annihilated by $I_0$.
		\item \( M \) is called \textit{finite free} if \( M \) is isomorphic to a finite direct sum of copies of \( A^a \).
	\end{enumerate}
\end{definition}

Obviously, finite free modules are almost finite free. Moreover, almost finite free modules are almost flat.

\begin{exam}[Almost free module but not free]\label{Ex: AF not F}

Let $K$ be a non-archimedean field with (multiplicative) norm $|\cdot|: K \to \mathbb{R}_{\geq 0}$. Put $\mathcal{O}_K$ the valuation ring and $\mathfrak m=\{x\in K: |x|<1\}$. Assume that the value group $\Gamma = |K^\times| \subset \mathbb{R}_{>0}$ is dense. Then, the pair $(\mathcal O_{K},\frak m)$ satisfies Assumption \ref{assum: basic setting of almost mathematics}. Suppose there exists an $r \in \mathbb{R}_{> 0}\backslash \Gamma$. Define
\[
I_r = \{x\in K: |x|<r\}.
\]
Then $I_r$ is an almost finite free but not finite free $\mathcal O_K^a$-module.

\end{exam}


\begin{exam}[Freeness does not implies projectiveness in $\mathbf{Mod}_A^a$]\label{ex: free not almost projective}
    In almost category, freeness does not imply projectiveness in the category $\mathbf {Mod}_A^a$.

For instance, working over the base ring $A=\mathcal{O}_{\mathbb{C}_{p}}$ with idempotent maximal ideal $\mathfrak{m}=\mathfrak{m}^{2}$. ${A}^{a}$ is free but not projective. Actually, the group $\operatorname{Ext}_{A^{a}}^{1}({A}^{a},A^{a})=\operatorname{Ext}_{A}^{2}(A/\mathfrak{m},\frak m)$, and is nonzero since $\mathbb{C}_{p}$ is not spherically complete. See \cite[Remark 4.2.5]{Bhatt2017Perfectoid} for a detailed explanation. 

 For some special $A$, there are even no non-trivial projective objects in the category $\mathbf{Mod}_A^a$ (\cite[Example 2.4.5]{Gabber2002AlmostRT}).

\end{exam}




\subsection{Almost vector bundles, locally almost free sheaves and locally free sheaves}

In this subsection, we fix a ring $A$ and an ideal $I \subseteq A$ satisfying Assumption \ref{assum: basic setting of almost mathematics}. All subsequent almost mathematics is considered relative to the pair $(A, I)$. 

The goal of this subsection is to define almost vector bundles.
It must be emphasized that in almost mathematics, since almost free modules and free modules are not always equivalent, a distinction also exists between almost vector bundles and locally free sheaves.

\begin{definition}\label{defi: almost sheaf of A mod}
    Let ${\mathcal{X}}$ be a site. An almost $A$-module sheaf (or an almost $A$-module presheaf) on ${\mathcal{X}}$ is defined as a sheaf (or presheaf) on ${\mathcal{X}}$ which takes values in the category $\mathbf{Mod}_{A}^a$.
\end{definition}

\begin{lemma}\label{Lem: Sh}
    Using the notation from Definition \ref{defi: almost sheaf of A mod}:
    \begin{enumerate}
        \item The natural embedding from the category of almost $A$-module sheaves to the category of almost $A$-module presheaves has an exact left adjoint, called sheafification \footnote{We ignore some set theoretic issues, as the sheafification may not exist if the category is too large. This can be fixed by shrinking the cardinality of $\mathcal X$}.
        \item The category of almost $A$-module sheaves is an abelian category satisfying (AB5), hence has enough injective objects.
    \end{enumerate}
\end{lemma}

\begin{proof}
    Via the functor $(-)_!$, identify $\mathbf{Mod}_A^a$ as a full additive abelian subcategory of $\mathbf{Mod}_A$. Consequently, almost $A$-module sheaves (or presheaves) can be viewed as a full additive subcategory of $A$-module sheaves (or presheaves). It is straightforward to verify that the almostification of the sheafification functor provides a left adjoint to the embedding of almost sheaves into almost presheaves. This proves the first statement. Since the category $\mathbf{Mod}_A^a$ is an abelian category satisfying (AB5) and is complete, the second statement can be proved using the same arguments as for sheaves of abelian groups.
\end{proof}

\begin{definition}\label{defi: almost ringed site setting}
    An almost ringed site is a pair $({\mathcal{X}}, {\mathcal{O}}_{{\mathcal{X}}})$, where ${\mathcal{X}}$ is a site and ${\mathcal{O}}_{{\mathcal{X}}}$ is a sheaf of $A^a$-almost algebras. For an almost ringed site $({\mathcal{X}}, {\mathcal{O}}_{{\mathcal{X}}})$, one can define ${\mathcal{O}}_{{\mathcal{X}}}$-almost module sheaves (or almost module presheaves) analogously to module sheaves (or presheaves). Denote by $\mathbf{Mod}_{{\mathcal{X}}}$ the category of ${\mathcal{O}}_{{\mathcal{X}}}$-module sheaves.
\end{definition}

After identifying $\mathbf{Mod}_A^a$ as a full additive abelian subcategory of $\mathbf{Mod}_A$ via $(-)_!$, it is not difficult to see that every $\mathbf{Mod}_A^a$-presheaf on any site admits a sheafification, and the sheafification functor is exact. Moreover, the category $\mathbf{Mod}_{{\mathcal{X}}}$ satisfies (AB5), hence possesses enough injective objects. Similar to ringed sites, one can define morphisms of almost ringed sites, and further define the pullback and pushforward of almost module sheaves.

For each almost $A$-module $M$ and an almost ringed site $({\mathcal{X}}, {\mathcal{O}}_{{\mathcal{X}}})$, define the almost sheaf $\widetilde{M}_{({\mathcal{X}},{\mathcal{O}}_{{\mathcal{X}}})}$ as the sheafification of the presheaf defined by $U \mapsto {\mathcal{O}}_{{\mathcal{X}}}(U) \otimes_{A^a} M$ (tensor product as almost $A$-modules).

\begin{definition}\label{def}
    Let $({\mathcal{X}}, {\mathcal{O}}_{{\mathcal{X}}})$ be an almost ringed site, and let ${\mathcal{M}}$ be an almost module sheaf on $({\mathcal{X}}, {\mathcal{O}}_{{\mathcal{X}}})$. We say ${\mathcal{M}}$ is:
    \begin{enumerate}
        \item An almost vector bundle of rank $r$, if for every object $U \in {\mathcal{X}}$ and every finitely generated ideal $I_0 \subseteq I$, there exists a covering $\{U_\lambda \to U : \lambda \in \Lambda\}$ such that on each open $U_{\lambda}$, there exists a morphism $({\mathcal{O}}_{{\mathcal{X}}}|_{U_\lambda})^{\oplus r} \to {\mathcal{M}}|_{U_\lambda}$ whose kernel and cokernel are annihilated by $I_0$.
        \item A locally free sheaf, if for every object $U \in {\mathcal{X}}$, there exists a covering $\{U_\lambda \to U : \lambda \in \Lambda\}$ such that on each $U_{\lambda}$, ${\mathcal{M}}|_{U_{\lambda}}$ is isomorphic to a finite direct sum of copies of ${\mathcal{O}}_{{\mathcal{X}}}$.
        \item A locally almost free sheaf, if for every object $U \in {\mathcal{X}}$, there exists a covering $\{U_\lambda \to U : \lambda \in \Lambda\}$ such that for each $U_{\lambda}$, there exists an almost free $A^a$-module $M$ satisfying \[{\mathcal{M}}|_{U_{\lambda}} \cong \widetilde{M}_{({\mathcal{X}}|_{U_{\lambda}}, {\mathcal{O}}_{{\mathcal{X}}}|_{U_{\lambda}})}.\]
    \end{enumerate}
\end{definition}

An important fact is that these properties are preserved under pullbacks.

\begin{lemma}\label{compatible}
	Suppose that $f:(\mathcal{X},\mathcal{O}_\mathcal{X})\to (\mathcal{Y},\mathcal{O}_\mathcal{Y})$ is a continuous morphism of almost ringed sites\footnote{Similar to ringed sites, a continuous morphism of almost ringed sites $f:(\mathcal{X},\mathcal{O}_\mathcal{X})\to (\mathcal{Y},\mathcal{O}_\mathcal{Y})$ consists of a morphism of sites $\mathcal{X}\to \mathcal{Y}$ in the sense of \cite[\href{https://stacks.math.columbia.edu/tag/00X1}{Tag 00X1}]{stacks-project} (denoted by $f^{-1}$ the morphism of categories $\mathcal{Y}\to \mathcal{X}$) and a morphism of almost ring sheaves $\mathcal{O}_\mathcal{Y}\to f_*\mathcal{O}_\mathcal{X}$.} satisfying that
    \begin{itemize}
        \item For any $U\in \mathcal X$, there exists a covering $\{U_\lambda\to U:\lambda\in\Lambda\}$ and a class of objects $\{V_\lambda:\lambda\in \Lambda\}$ such that $\mathrm{Mor}_\mathcal{X}(U_\lambda,f^{-1}V_\lambda)\neq \emptyset$.
    \end{itemize}
	Then for any almost vector bundle (resp. locally free sheaf, locally almost free sheaf) $\mathcal{M}$ on $(\mathcal{Y},\mathcal{O}_\mathcal{Y})$, $f^*\mathcal{M}$ is an almost vector bundle (resp. locally free sheaf, locally almost free sheaf) on $(\mathcal{X},\mathcal{O}_\mathcal{X})$.
\end{lemma}

\begin{proof}
We will only prove the case of almost vector bundles; the proofs for the other two cases are entirely analogous. Suppose $\mathcal{M}$ is an almost vector bundle on $\mathcal{Y}$. For $U\in\mathcal{X}$ and a finitely generated ideal $I_{0}\subseteq I$, we need to show that there exists a cover $\{U_{\lambda}\to U:\lambda\in\Lambda\}$ of $U$ such that on each $U_{\lambda}$, there exists a positive integer $r$ and a morphism
\[
\nu_{\lambda}:\mathcal{O}_{\mathcal{X}}|_{U_{\lambda}}^{\oplus r}\to f^{*}\mathcal{M}|_{U_{\lambda}}
\]
such that the kernel and cokernel of $\nu_{\lambda}$ are annihilated by $I_{0}$. 

By assumption, we can choose a cover $\{U_{\lambda}\to U:\lambda\in\Lambda\}$, objects $\{V_{\lambda}\in\mathcal{Y}:\lambda\in\Lambda\}$, and morphisms $\{\sigma_{\lambda}:X_{\lambda}\to f^{-1}V_{\lambda}\}$ in $\mathcal{X}$. For each $V_{\lambda}$, we can choose a cover $\{V_{\lambda,\gamma}\to V_{\lambda}:\gamma\in\Gamma_{\lambda}\}$ such that on each $V_{\lambda,\gamma}$, there exists a morphism $\mu_{\lambda,\gamma}:\mathcal{O}_{\mathcal{Y}}|_{V_{\lambda,\gamma}}^{\oplus r}\to\mathcal{M}|_{V_{\lambda,\gamma}}$ such that both $\ker\mu_{\lambda,\gamma}$ and $\operatorname{coker}\mu_{\lambda,\gamma}$ are annihilated by $I_{0}$.

For any $\lambda\in\Lambda$, $\gamma\in\Gamma_{\lambda}$, define $U_{\lambda,\gamma}=f^{-1}V_{\lambda,\gamma}\times_{f^{-1}V_{\lambda,\sigma}}U_{\lambda}$. Then, by definition, $\{U_{\lambda,\gamma}\to U:\lambda\in\Lambda,\gamma\in\Gamma_{\lambda}\}$ is a cover. Pulling back $\mu_{\lambda,\gamma}$ along the natural morphism $U_{\lambda,\gamma}\to f^{-1}V_{\lambda,\gamma}$, it follows that on $U_{\lambda,\gamma}$ there exists a morphism $\nu_{\lambda,\gamma}:\mathcal{O}_{\mathcal{X}}|_{U_{\lambda,\gamma}}^{\oplus r}\to f^{*}\mathcal{M}|_{U_{\lambda,\gamma}}$ whose kernel and cokernel are both annihilated by $I_{0}$.
\end{proof}

We end this subsection with an example illustrating that an $\mathcal{O}_X^+$-locally free sheaf in the \'{e}tale topology does not necessarily descend to an $\mathcal{O}_X^+$-locally free sheaf in the analytic topology.

\begin{exam}[\'Etale vector bundle which is not analytic vector bundle]\label{ex: et bundle not descent}
Consider a finite ramified Galois extension $L/K$ of perfectoid fields. Then, $L^{\circ}$ is almost free but need not to be free over $K^{\circ}$. For example, take  $K = \widehat{\mathbb{Q}_p(p^{1/p^\infty})}$, and $L = K(\sqrt{p})$ with $p \neq 2$ (cf. \cite[Example 4.3.2]{Bhatt2017Perfectoid}). Let $X=\spa(K,K^\circ)$ and $Y=\spa(L,L^{\circ})$. Then $L^\circ$, considered as a $K^{\circ,a}$-module is an almost vector bundle on $X$ which is not a locally free sheaf. However, the \'etale sheaf $\CV$ on $X$ induced by $L^\circ$ is locally free. Indeed, since $L/K$ is Galois, the almost tensor product $L^{\circ,a}\otimes_{K^{\circ,a}}L^{\circ,a}$ is isomorphic to $\bigoplus_{g\in \mathrm{Gal}(L/K)} L^{\circ,a}$, meaning that $\CV|_{Y}$ is free.
\end{exam}

\section{$V$-descent of Almost Vector Bundles}

In this section, we always fix an affinoid perfectoid space $X=\Spa(A,A^+)$, and deal with the almost mathematical relationship with respect to $A^{\circ\circ}$. Fix a perfectoid pseudo-uniformizer $\pi\in A^+$, such that $\pi^{1/p^n}\in A^+$  for $\forall n\in\mathbb{N}$. 
Then we have $A^{\circ\circ}=\bigcup_{n\ge 0}\pi^{1/p^n}A^+$ by Example \ref{EX: IPR}.

We denote by $X_{\mathrm{an}}$ the analytic site of $X$, by $X_v$ the $v$-site of $X$ (cf. \cite[Definition 17.1.1]{Scholze_2020}).  It turns out that  both  \( X_{\mathrm{an}}^{+,a}:=(X_{\mathrm{an}}, \mathcal{O}^{+,a}_X)  \) and \( X_v^{+,a}:= (X_v, \mathcal{O}^{+,a}_X) \) are almost ringed sites in sense of Definition \ref{defi: almost ringed site setting}, by \cite[Theorem 17.1.3]{Scholze_2020} and \cite[Theorem 6.3]{scholze2012perfectoid}. 

To make our proof more concise, in this section, we use the functor $(-)_{!}$ to view all almost $A^{+,a}$-module sheaves as $A^{+}$-module sheaves. According to Lemma \ref{embed}, for a site $\mathcal{X}$, an almost $A^{+,a}$-module sheaf on $\mathcal{X}$ is equivalent to an $A^{+}$-module sheaf whose all sections satisfy \[A^{\circ\circ}\otimes_{A^+}M\cong M.\] In particular, for an almost vector bundle $\mathcal{V}$ on $X_{\mathrm{an}}^{+,a}$ (resp. $X_{v}^{+,a}$), it can be naturally viewed as a subsheaf of $\mathcal{O}_X^+$-modules of the analytic (resp. $v$-) vector bundle $\mathcal{V}[\frac{1}{\pi}]$ on $X$.

Fix an integer $N\geq 1$. We begin our proof by showing that under the $v$-topology, almost vector bundles are locally almost free.

\begin{theorem}\label{Thm: vb is svb}
Any almost vector bundle of rank $N$  over $X_v^{+,a}$ is locally almost free in sense of Definition \ref{def}.
\end{theorem}

\begin{proof}
Let  $\CV$ be an almost vector bundle of rank $N$ over $X_v^{+,a}$. We need the following lemma.
\begin{lemma}\label{Lem: construct}
    There exists a tower of affinoid perfectoid spaces $X=X_0 \xleftarrow{p_1} X_1 \xleftarrow{p_2} \cdots$ such that 
    \begin{enumerate}[(\roman*)]
    \item  For each $n\ge1$, $X_{n}$ is a v-cover of $X_{n-1}$.
    \item For each $n$, there exists  an injective $f_n: (\mathcal{O}^{+}_{X_{n}})^{\oplus N} \hookrightarrow \mathcal{V}|_{X_n}$ satisfying $\pi^{1/p^n} \cdot \text{coker}(f_n) = 0$, and sheaf morphisms $\iota_n:(\CO^{+}_{X_{n-1}})^{\oplus N}|_{X_n}\to (\CO^{+}_{X_{n}})^{\oplus N}$such that the following diagram commutes:
      \begin{center}
\begin{tikzcd}
(\CO^{+}_{X_{n-1}})^{\oplus N}|_{X_n} \arrow[dd, "\iota_n"'] \arrow[rd, "f_{n-1}|_{X_n}"] &            \\
                                                                             & \CV|_{X_n} \\
(\CO^{+}_{X_{n}})^{\oplus N} \arrow[ru, "f_n"']                                         &           
\end{tikzcd}
    \end{center}
    \end{enumerate}
\end{lemma}

\begin{proof}[Proof of Lemma \ref{Lem: construct}]
We prove the lemma by induction on $n$.
For $n=1$, by Definition \ref{def}, there exists a $v$-cover of $X_1\to X$ and an injective $\mathcal O^+_{X_1}$-linear homomorphism $f_1:(\mathcal O_{X_1}^{+})^{\oplus N}\to \CV|_{X_1}$ whose cokernel killed by $\pi^{\frac{1}{p}}$.
Assume now that we have been constructed everything for $n-1$. Since $\CV|_{X_{n-1}}$ is an
almost vector bundle over $X_{n-1}$ by Lemma \ref{compatible}, there exists a
$v$-cover $Y_n\to X_{n-1}$
together with an injective $\mathcal{O}_Y^+$-linear morphism $h:(\CO_{Y_n}^{+})^{\oplus N}\to \CV|_{Y_n}$
such that $\pi^{1/p^n}\cdot\coker(h)=0$.

Denote by $V=\Gamma(Y_n,\mathcal V)$. Then by Theorem \ref{theo: rational v descent}, $V[\frac{1}{\pi}]=\Gamma(Y,\mathcal{V}[\frac{1}{\pi}])$ is a rank $N$ projective $A$-module. By Lemma \ref{free}, there exists a $v$-cover $X_n=\Spa(A_n,A_n^+)\to Y_n$ such and a free $A_n^+$-submodule $F_n\subseteq (A_n^+\otimes_{A^+}V)^{\pi-\text{tf}}$ of rank $N$ containing the images of both $h$ and $f_{n-1}$. 

Define the morphism $f_n: (\CO^{+}_{X_{n}})^{\oplus N} \to \CV|_{X_n}$ via the inclusion $F_{n} \hookrightarrow \Gamma(X_n, \CV|_{X_n})$. 
Furthermore, the inclusion of the image of $f_{n-1}$ into $F_n$ induces the injection $\iota_n$ and the inclusion of the image of $h$ into $F_n$ implies  $\pi^{1/p^n}\cdot\coker(f_n)=0$.
\end{proof}

Back to the proof of Theorem \ref{Thm: vb is svb}, let
$ X_\infty := \varprojlim X_n$
which is obviously affinoid perfectoid (denoted by $\Spa(A_{\infty},A^+_{\infty})$) and is a $v$-cover of $X$. Then, Lemma \ref{Lem: construct} implies that for any $\alpha\in A^{\circ\circ}$, there exists an injective homomorphism 
\[(\CO_{X_{\infty}}^{+})^{\oplus N}\to \CV|_{X_{\infty}}\]
whose cokernel killed by $\alpha$. Then it follows that, for any affinoid perfectoid space $Y=\Spa(B,B^+)/X_{\infty}$,
\[B^+\otimes_{A_{\infty}^+}\Gamma(X_\infty,\CV)\to \Gamma(Y,\CV)\]
is an almost isomorphism. The theorem then follows.
\end{proof}

Now, we proceed to prove Theorem \ref{theo: v descent of almost vb}. 
Since the sheafification argument for $\nu^*$ involves handling various descent data of $v$-sheaves, we need the following proposition. Note the complete tensor product is well defined  since the category $\mathbf{Mod}_{A^+}^a$ is both complete and cocomplete according to \cite[Corollary 2.2.15.]{Gabber2002AlmostRT}.
\begin{proposition}\label{TECH}
    Let \(f:Y=\mathrm{Spa}(B,B^+)\rightarrow X=\mathrm{Spa}(A,A^+)\) be a \(v\)-cover of affinoid perfectoid spaces. Given an almost finite free \(B^{+,a}\)-module \(V\) of rank \(N\) and a \((B^{+}\widehat{\otimes}_{A^{+}}B^{+})^a\)-module isomorphism
    \[
    \psi : B^{+,a} \widehat{\otimes}_{A^{+,a}} V \xrightarrow{\sim} V \widehat{\otimes}_{A^{+,a}} B^{+,a}
    \]
    that satisfies the cocycle condition, then for any \(\alpha \in A^{\circ\circ}\), there exists an analytic cover \(\{\mathrm{Spa}(A_i,A_i^+)\to \mathrm{Spa}(A,A^+)\}_{i\in I}\) such that for each \(i\in I\):
    \begin{enumerate}[(\roman*)]
        \item There exists an \(A_i^{+,a}\)-module morphism
        \[
        g_i: (A_i^{+,a})^N \to W_{A_i^{+}}:=\{v\in V_i \mid \psi_{{A_i^{+}}}(1\widehat{\otimes}_{A_i^{+,a}}v)=v\widehat{\otimes}_{A_i^{+,a}}1\}
        \]
        such that \(\alpha\cdot \mathrm{coker}(g_i)=0\), where \(\psi_{A_i^+}\) is the base change of \(\psi\) to \(A_i^{+,a}\) and \(V_i := V\widehat{\otimes}_{A^{+,a}}A_i^{+,a}\).

        \item The natural \(B_i^{+,a}\)-module morphism
        \[
        h_i:W_{A_i^+}\widehat{\otimes}_{A_i^{+,a}} B^{+,a}_i\to V_i
        \]
        satisfies \(\alpha\cdot \mathrm{coker}(h_i)=0=\alpha\cdot \mathrm{ker}(h_i)\), where \(B_i^+:=B^+\widehat{\otimes}_{A^+} A_i^+\).
    \end{enumerate}
    Note that each \(\ker g_i\) is automatically zero.
\end{proposition}

\begin{proof}
    We first prove the case where $X=\Spa(K,K^+)$ where $K$ is a perefctoid field and $K^+$ is a valuation ring containing $K^{\circ\circ}$. This is indeed an almost version of completely faithfully flat descent.
    
Note that $f$ induces a ring morphism $f^\circ:K^\circ \to B^\circ$, which is obviously faithfully flat since $K^\circ$ is a valuation ring of rank $1$.
By the complete faithfully flat descent, $(V,\psi)$ descend to a $\pi$-complete and $\pi$-completely flat module over $K^\circ$, which is denoted by $W$. We only need to prove that $W$ is almost free of rank $N$. By Theorem \ref{theo: rational v descent}, $W[\frac{1}{\pi}]$ is a vector space of dimension $N$ and hence the claim is already included in the proof of \cite[Theorem 3.3]{yu2025deltalifting1dimensionalanalyticfields}

Then we turn to general case.
For any  $\beta\in A^{\circ\circ}$, there exists a homomorphism $u: (B^{+})^N \hookrightarrow V$ s.t. $\beta\cdot \coker (u)=0$. Let 
$$p_1, p_2:   Y \times_X Y \to Y$$ 
be the first and second canonical projections. We pull back  $u$ via $p_1$ and $p_2$ and note 
the descent datum $\psi$ is an $ B^{+}\widehat\otimes_{A^{+}} B^{+}$-module isomorphism from $p_2^*V$ to $p_1^*V$.
\begin{center}
\begin{tikzcd}
(B^{+}\widehat{\otimes}_{A^{+}}B^{+})^{N} \arrow[rr, "p^*_2u", hook] &  & p^*_2V=B^{+}\widehat{\otimes}_{A^{+}}V \arrow[d, "{\psi,\cong}"] \\
(B^{+}\widehat{\otimes}_{A^{+}}B^{+})^{N} \arrow[rr, "p^*_1u", hook]                         &  & p^*_1V=V\widehat{\otimes}_{A^{+}}B^{+}                     
\end{tikzcd}
\end{center}
Thus $\psi$ induces some $G\in \mathrm{GL}_N(B\widehat\otimes_AB)$ (Note $G$ may not send $(B^{+}\widehat{\otimes}_{A^{+}}B^{+})^{N}$ to itself).

As $\coker(u)$ is killed by $\beta$, elements of both $\beta G$ and $\beta G^{-1}$ lie in $B^{+}\widehat{\otimes}_{A^{+}}B^+$.

Let $p_{ij}: Y \times_X Y \times_X Y \to Y \times_X Y$ for $1 \leq i < j \leq 3$ denote the canonical projection onto the $i$-th and $j$-th factors. The cocycle condition for $\psi$ implies
\[
    p_{12}^*G \cdot p_{23}^*G - p_{13}^*G = 0.
\]

We claim that for any $x\in X$, there exists an affinoid perfectoid open neighborhood $x\in U=\Spa(A_1,A_1^+)\subseteq X$ and a matrix $S_U\in \mathrm{GL}_N(A_1)$ such that 
\begin{enumerate}[(1)]
    \item $p_1^*S_U^{-1}\cdot p_2^*S_U=G$ as matrices with coefficients in $B\widehat{\otimes}_AB\widehat{\otimes}_AA_1$;
    \item elements of both $\beta S_U$ and $\beta S_U^{-1}$ lie in $A_1$.
\end{enumerate}
Obviously, this claim implies the theorem since $X$ is quasicompact.

Indeed, apply $-\widehat{\otimes}_{A^{+,a}}k(x)^{+,a}$ on everything, we see 
$\psi_x:B^{+,a}_x\widehat{\otimes}_{k(x)^{+,a}}V_x\to V_x\widehat{\otimes}_{k(x)^{+,a}}B^{+,a}_x$
satisfies the cocycle condition where $B_x^{+}:=B^{+}\widehat{\otimes}_{A^{+}}k(x)^{+}$, and $V_x:=V\widehat{\otimes}_{A^{+,a}}k(x)^{+,a}$. 
Let $G_x$ be the base change of $G$ to $k(x)$. Then the case of adic points implies that there exists an $S_x\in \mathrm{GL}_N(B_x)$ such that
\[p_1^*S_x^{-1}\cdot p_2^{*}S_x=G_x,\]
and that elements of $\beta S_x$ adn $S_x^{-1}$ lies in $B_x^+$.

Approximate $S_x$ by some $S_U \in \mathrm{GL}_N\big(B \widehat{\otimes}_{A} \mathcal{O}_{X}(U)\big)$ for a rational neighborhood $U \subset X$ of $x$ in $X$ so that all elements of $\beta S_U$ and $\beta S_U^{-1}$ lie in $\CO_X^+(U)$ and that \[H_U:=p_1^*S_U\cdot G\cdot p_2^*S_U^{-1}\equiv 1 \pmod{\pi}.\] 
Replace $X$ by $U$ and let $C$ satisfies $H=1+\pi C$. Then by the cocycle condition inherited from $G$ on $H$ we have
\[
(1+\pi p^*_{12}C)(1+\pi p^*_{23}C)\equiv(1+\pi p^*_{13}C) \mod \pi^2.
\]
This implies that the projection of $C$ in $\CO_X^+(U)/\pi$ lies in $[\check{H}^1(Y/X, M_r(\mathcal{O}_X^{+}/\varpi))]^a$, which is zero in $\mathbf{Mod}_{A^+}^a$ by  \cite[Theorem 17.1.3]{Scholze_2020}. Hence, we can conjugate $H$ by a matrix in $1 + \varpi^{1-\epsilon} M_r(B^{+,a})$ so as to assume that $H \equiv 1 \pmod{\varpi^{2-\epsilon}}$ for any $\epsilon > 0$. The claim then follows from an approximation argument on $H$.
\end{proof}

Now, we can prove Theorem \ref{theo: v descent of almost vb}, which we restate here.
\begin{theorem}
\label{SAAVB}
Given an affinoid perfectoid space $X=\Spa(A,A^+)$, and denote by 
 \[
\nu: X^{\sim}_\text{v} \to X^{\sim}_\text{an}.
 \]
 the natural projection.
Then, the functor $\nu_*$ induces an equivalence of categories:
\[
\nu_*:\{\text{Almost vector bundles on } X_{{v}}^{+,a}\}
\simeq
\{\text{Almost vector bundles on }X_{{an}}^{+,a}\}
\]
\end{theorem}

\begin{proof}

By Lemma \ref{compatible}, for any almost vector bundle $\CV$ over $X_{\mathrm{an}}^{+,a}$, $\nu^*\CV$ is an almost vector bundle. Moreover, since $R\nu_{*}\CO_{X_v}^{+,a}=\CO_{X_{\mathrm{an}}}^{+,a}$, the natural morphism $\CV\to \nu_{*}\nu^*\CV$ is an isomorphism.

It remains to prove that for any almost vector bundle $\CV$ on $X_{v}^{+,a}$, $\nu_*\CV$ is an almost vector bundle on $X_{\mathrm{an}}^{+,a}$ and that the natural morphism
\[\nu^*\nu_*\CV\to \CV\]
is an almost isomorphism.

In fact, by Theorem \ref{Thm: vb is svb}, we can choose an affinoid perfectoid $v$-cover $Y=\Spa(B,B^+)\to X$ such that $\CV|_{Y}$ is induced by an almost free module $V/B^{+,a}$. For any $\alpha \in A^{\circ\circ}$, Proposition \ref{TECH} provides an analytic cover $\{\mathrm{Spa}(A_i, A_i^+) \to \mathrm{Spa}(A, A^+)\}_{i \in I}$ such that, for each $i \in I$, there exists an injective $\CO_{X_i}^{+,a}$-linear morphism of almost sheaves \[u:\big(\CO_{X_i}^{+,a}\big)^{\oplus N}\to \nu_*\CV|_{X_i}\]
satisfying that the cokernel of the morphism on $X_v$ induced by $u$ is killed by $\alpha$. This shows that $\coker(u)$ is killed by $\alpha$ and both the kernel and cokernel of the natural morphism
\[\nu^*\nu_*\CV\to \CV\]
is killed by $\alpha$. The theorem then follows immediately.
\end{proof}

\section{$V$-local Structure of Almost Vector Bundles}\label{Sec: Spherical completeness}
In this section, we will use some results on spherical completeness to prove the following theorem. 

\begin{theorem}\label{Thm: SC over v}
Given an  affinoid perfectoid space $X=\Spa(A,A^+)$ and an almost vector bundle $\CV$ of rank $N$ over $\CO_X^{+,a}$ under the v-topology, there exists a v-cover $\{U_\lambda \to U : \lambda \in \Lambda\}$ such that for each $U_{\lambda}$ there is an almost isomorphism:
\[
f:(\CO^{+,a}_{U_{\lambda}})^{\oplus N} \to \CV|_{U_\lambda}
\]
\end{theorem}

This theorem implies that an almost vector bundle under the $v$-topology is a locally free sheaf, in the sense of Definition \ref{def}. The proof of this theorem is at the end of this section.

\subsection{Properties of spherical completeness}

The key to proving Theorem \ref{Thm: SC over v} lies in the observation that for a spherically complete non-Archimedean field $K$,  finite-rank almost free modules over $K^{+,a}$ admit a well-behaved decomposition theorem.
In this subsection, fix a spherically non-Archimedean field $K$ and denote by $|\cdot|$ its norm. The following computation is well-known (see, for example \cite[Remark 4.2.5]{Bhatt2017Perfectoid})

\begin{proposition}\label{prop: ext of almost free modules sph cpt}
We have $\Ext^1_{K^{\circ}}(K^{\circ\circ}, K^{\circ\circ})=0.$
\end{proposition}

\begin{proof}
Throughout this proof, all modules and homomorphisms are considered over the base ring $K^\circ$ unless otherwise specified. Consider the short exact sequence of $K^\circ$-modules:
\[
0 \to K^{\circ\circ} \to K \to K/K^{\circ\circ} \to 0.
\]
Applying the functor $\Hom_{K^\circ}(K^{\circ\circ}, -)$ yields:
\[
0 \to \Hom(K^{\circ\circ}, K^{\circ\circ}) \to \Hom(K^{\circ\circ}, K)  \xrightarrow{\hspace{0.3cm} p \hspace{0.3cm}} \Hom(K^{\circ\circ}, K/K^{\circ\circ}) \to \Ext^1(K^{\circ\circ}, K^{\circ\circ}) \to \Ext^1(K^{\circ\circ}, K).
\]
Since $K$ is the fraction field of the valuation ring $K^\circ$ thus $\Ext^1(K^{\circ\circ}, K) = 0$. Consequently, we have a canonical identification:
\[
\Ext^1(K^{\circ\circ}, K^{\circ\circ}) \cong \Hom(K^{\circ\circ}, K/K^{\circ\circ}) / \im(p).
\]
To show the vanishing of $\Ext^1$, it suffices to prove that the map $p$, induced by the natural projection $K \to K/K^{\circ\circ}$, is surjective. 

We first observe that $\Hom_{K^\circ}(K^{\circ\circ}, K) \cong K$. Indeed, any $f \in \Hom(K^{\circ\circ}, K)$ is uniquely determined by the image of any non-zero element (e.g., $f(\pi)$), as $f(ax) = af(x)$ for all $a \in K^\circ$.

Now, let $f \in \Hom(K^{\circ\circ}, K/K^{\circ\circ})$ be an arbitrary homomorphism. For each $n \geq 0$, choose a representative $x_n \in K$ of $f(\pi^{1/p^n})$. The $K^\circ$-linearity of $f$ implies that for all $n \geq 1$:
\[
x_{n-1} \equiv \pi^{1/p^{n-1} - 1/p^n} x_n \pmod{K^{\circ\circ}},
\]
which is equivalent to the condition $|\pi^{1/p^{n-1}} x_{n-1} - \pi^{1/p^n} x_n| < |\pi|^{1/p^{n-1}}$. Multiplying by $|\pi|^{1 - 1/p^{n-1}}$, we obtain:
\[
|\pi^{1 - 1/p^{n-1}} x_{n-1} - \pi^{1 - 1/p^n} x_n| < |\pi|^{1 - 1/p^{n-1}}.
\]
For each $n$, define the closed ball $B_n := \{ z \in K \mid |z - \pi^{1 - 1/p^n} x_n| \leq |\pi|^{1 - 1/p^n} \}$. The inequality above demonstrates that $\pi^{1 - 1/p^n} x_n \in B_{n-1}$. Since the radii $r_n = |\pi|^{1 - 1/p^n}$ form a non-increasing sequence (as $|\pi| < 1$ and $1-1/p^n$ increases), we have a nested sequence of balls $B_0 \supseteq B_1 \supseteq B_2 \supseteq \dots$. By the spherical completeness of $K$, the intersection $\bigcap_{n \geq 0} B_n$ is non-empty.

Pick an element $x \in \bigcap_{n \geq 0} B_n$. We define a lift $\tilde{f} \in \Hom(K^{\circ\circ}, K)$ by setting $\tilde{f}(\pi) = x$ and, more generally, $\tilde{f}(\pi^{1/p^n}) = \pi^{1/p^n-1}x$. By the definition of $B_n$, we have $|x - \pi^{1 - 1/p^n} x_n| \leq |\pi|^{1 - 1/p^n}$. In fact, since $x \in B_{n+1}$, the ultrametric inequality ensures the strict inequality $|x - \pi^{1 - 1/p^n} x_n| < |\pi|^{1 - 1/p^n}$ holds. Dividing by $|\pi|^{1 - 1/p^n}$, we obtain:
\[
|\pi^{1/p^n - 1}x - x_n| < 1,
\]
which implies $\tilde{f}(\pi^{1/p^n}) \equiv x_n \pmod{K^{\circ\circ}}$. Thus, $p(\tilde{f}) = f$, establishing the surjectivity of $p$.
\end{proof}

\begin{theorem}\label{Thm: al free is free}
Additionally, assume that the value group of $K$ is $\mathbb{R}_{>0}$. Then, every almost free $K^{+,a}$-module $M$ is finitely free in $\mathbf{Mod}_{K^+}^a$.
\end{theorem}

\begin{proof}
It is enough to show that $M_!$ is almost isomorphic to a free module.

First, we prove for an almost free module of rank $1$. By choosing a non-zero element $m\in M_!$, we may consider $M_!$ as a $K^{\circ}$-submodule of $K$. Since $|K^\times|=\mathbb{R}_{>0}$, there exists $k_0\in K^\times$ such that
\[|k_0|=\sup_{m\in M}|m|.\]
It is easy to check that $M_!=K^{\circ\circ}k_0$, which is free as an almost module.

The general case follows from an induction on the rank $N$ and Proposition \ref{prop: ext of almost free modules sph cpt}.
\end{proof}

\subsection{Spherical completeness and the structure of almost vector bundles}
For a non-archimedean field $K$, we can follow the construction of a universal field in \cite{robert2000course} to obtain a spherically complete field $\widetilde{K}$. Thus, we have the following lemma.

\begin{lemma}\label{Lem: extension}
For any non-archimedean field $K$ with a rank $1$ valuation corresponding to $\CO_K$, there exists a field extension of $K$, denoted by $\widetilde{K}$, which is perfectoid and spherically complete with respect to a rank $1$ valuation. Furthermore, for a valuation $|\cdot|$ on $K$ with valuation ring $K^+$, this valuation induces a valuation on $\widetilde{K}$ whose valuation group is $\mathbb{R}_{>0}$.
\end{lemma}

\begin{proof}
Following the construction in \cite[Section 2.2]{robert2000course}, there exists a spherically complete field $\widetilde{K}$ associated with $K$. Since $K$ contains $\mathbb{Q}_p$, it follows from \cite[Proposition 2.2.3]{robert2000course} that value group of  $\widetilde{K}$ is $\mathbb{R}_{>0}$.
\end{proof}

Although $\widetilde{K}$ may be made perfectoid via an appropriate field extension, a more general result can be established.
\begin{theorem}\label{perfectoid}
Let $K$ be a spherically complete non-Archimedean field. Assume that:
\begin{enumerate}[(\roman*)]
    \item The value group $\Gamma$ is p-divisible;
    \item The residue field $k$ of $K$ is a perfect field.
\end{enumerate}
Then $K$ is a perfectoid field.
\end{theorem}

Before the proof, we recall a non-Archimedean version of the Hahn-Banach theorem due to Ingleton.
\begin{theorem}[{\cite[Theorem 3]{Ingleton_1952}}]\label{Hahn}
Let $K$ be a spherically complete field and $V$ a Banach space over $K$. If $W$ is a subspace of $V$ with a linear operator  
\[
f:W\to K, \|f\|\le 1,
\]
then there exists a linear functional $\widetilde{f}: V \to K$ extending $f$ such that $$\|\widetilde{f}\| = \|f\|\quad \text{and}\quad \widetilde{f}|_W = f.$$
\end{theorem}

\begin{proof}[Proof of Theorem \ref{perfectoid}]
Fix a pseudo-uniformizer $\varpi \in \mathcal{O}_K$ satisfying $p \in \varpi^p \mathcal{O}_K$. To show that $K$ is perfectoid, it suffices to prove that the Frobenius morphism
\[
\varphi_K : \mathcal{O}_K/\varpi \to \mathcal{O}_K/\varpi^p, \quad x \mapsto x^p
\]
is surjective.

Let $L = \widehat{\overline{K}}$ be the completion of the algebraic closure of $K$, which is a perfectoid field. Regarding $L$ as a $K$-Banach space, Theorem \ref{Hahn} ensures the existence of a $K$-linear functional $f: L \to K$ extending the identity map $\mathrm{id}_K: K \to K$ such that $\|f(a)\| \le \|a\|$ for all $a \in L$. The norm condition implies $f(\mathcal{O}_L) \subseteq \mathcal{O}_K$, whence $f$ induces a morphism of $\mathcal{O}_K$-modules $\overline{f}: \mathcal{O}_L/\varpi \to \mathcal{O}_K/\varpi$.

We define a map $\sigma: \mathcal{O}_K/\varpi^p \to \mathcal{O}_K/\varpi$ as the composition in the following commutative diagram:
\begin{center}
\begin{tikzcd}
\mathcal{O}_K/\varpi^p \arrow[r, "\iota"] \arrow[d, "\sigma"'] & \mathcal{O}_L/\varpi^p \arrow[d, "\varphi_L^{-1}"] \\
\mathcal{O}_K/\varpi & \mathcal{O}_L/\varpi \arrow[l, "\overline{f}"] 
\end{tikzcd}
\end{center}
where $\iota$ is the natural inclusion and $\varphi_L^{-1}$ is the inverse of the Frobenius morphism on $L$ (which exists since $L$ is perfectoid). By construction, $\sigma$ satisfies $\sigma(x^p y) = x \sigma(y)$ for $x \in \mathcal{O}_K/\varpi$ and $\sigma(x^p) = x$.

Let $\mathfrak{m}$ denote the maximal ideal of $\mathcal{O}_K$. If $\varphi_K$ were not surjective, then $\sigma$ would not be injective, as $\sigma \circ \varphi_K = \mathrm{id}$. Thus, there would exist a non-zero element $x \in \mathcal{O}_K/\varpi^p$ such that $\sigma(x) = 0$.

We claim that any $x \in \mathcal{O}_K/\varpi^p$ can be written as $x = y^p(1+u)$ for some $y \in \mathcal{O}_K/\varpi$ and $u \in \mathfrak{m}/\varpi^p$. For convenience, we lift these elements to $\mathcal{O}_K$ without changing notation. Since $\Gamma$ is $p$-divisible, we may write $x = z^p v$ for some $v \in \mathcal{O}_K^\times$. Let $\overline{v} \in k$ be the image of $v$ under the residue map. Since $k$ is perfect, there exists $y' \in k$ such that $(y')^p = \overline{v}$. Let $w \in \mathcal{O}_K$ be a lift of $y'$. Setting $u' = v - w^p \in \mathfrak{m}$, the ultrametric inequality implies $|w|^p = |v| = 1$, so $w \in \mathcal{O}_K^\times$. Thus $v = w^p(1 + u)$ where $u = u'/w^p \in \mathfrak{m}$. Setting $y = wz$ yields the desired decomposition.

Returning to the proof, $\sigma(x) = 0$ implies $y \sigma(1+u) = 0$. Since $x \neq 0$ implies $y \neq 0$, we must have $\sigma(1+u) = 0$. Applying the same decomposition to $u \in \mathfrak{m}/\varpi^p$, we can write $u = u_1^p(1+u_2)$ with $u_1\in \CO_K/\varpi$ and $u_2\in \mathfrak{m}/\varpi^p$. Because $1+u_2$ is invertible and $u\in \mathfrak{m}/\varpi^p$, we have $u_1\in \mathfrak{m}/\varpi$. Then $\sigma(u) = u_1 \sigma(1+u_2) \in \mathfrak{m}/\varpi$. It follows that $\sigma(1+u) = 1 + \sigma(u)$ is a unit in $\mathcal{O}_K/\varpi$, contradicting the fact that $\sigma(1+u) = 0$. Therefore, $\varphi_K$ must be surjective.
\end{proof}

Recall that for finitely presented modules, tensor products commute with direct products. The following lemma establishes a corresponding result for almost free modules in the perfectoid setting.

\begin{lemma}\label{lemma of commut}
Let $\Spa(A,A^+)$ be an affinoid perfectoid space, and let $M$ be an almost free $A^{+,a}$-module of rank $N$. For any family $\{N_i\}_{i \in I}$ of complete $A^{+,a}$-modules, there is a natural almost isomorphism:
\[
M \widehat{\otimes}_{A^{+,a}} \prod_{i \in I} N_i \cong \prod_{i \in I} (M \widehat{\otimes}_{A^{+,a}} N_i).
\]
\end{lemma}

\begin{proof}
We first reduce the claim to the case of algebraic (non-completed) tensor products. By the definition of an almost free module, for each $n \geq 1$, there exists an injective morphism of $A^{+,a}$-modules
\[
f_n : (A^{+,a})^N \to M
\]
whose cokernel is annihilated by $\pi^{1/p^n}$. Consider the following commutative diagram:
\begin{center}
\begin{tikzcd}
(A^{+,a})^N \otimes_{A^{+,a}} N_i \arrow[r] \arrow[d, "\cong"'] & M \otimes_{A^{+,a}} N_i \arrow[d] \\
(A^{+,a})^N \widehat{\otimes}_{A^{+,a}} N_i \arrow[r] & M \widehat{\otimes}_{A^{+,a}} N_i
\end{tikzcd}
\end{center}
Since $N_i$ is complete, the left vertical map is an isomorphism. The kernel and cokernel of the horizontal maps are annihilated by $\pi^{1/p^n}$ because $\coker(f_n)$ is so. By letting $n \to \infty$, we deduce a natural almost isomorphism $M \otimes_{A^{+,a}} N_i \cong M \widehat{\otimes}_{A^{+,a}} N_i$. Thus, it suffices to show that
\[
M \otimes_{A^{+,a}} \prod_{i \in I} N_i \cong \prod_{i \in I} (M \otimes_{A^{+,a}} N_i)
\]
Consider the following commutative diagram with exact rows, where $g_n$ and $h_n$ are induced by $f_n$, and we denote $C_n := \coker(f_n)$:
\begin{center}
\begin{tikzcd}
(A^{+,a})^N \otimes \prod N_i \arrow[r, "h_n"] \arrow[d, "\cong"'] & M \otimes \prod N_i \arrow[r] \arrow[d, "\sigma"] & C_n \otimes \prod N_i \\
\prod \left( (A^{+,a})^N \otimes N_i \right) \arrow[r, "g_n"] & \prod (M \otimes N_i) \arrow[r] & \prod (C_n \otimes N_i)
\end{tikzcd}
\end{center}
The terms $C_n \otimes \prod N_i$ and $\prod (C_n \otimes N_i)$ are both annihilated by $\pi^{1/p^n}$. Furthermore, the kernels of $h_n$ and $g_n$ are images of $\Tor_1^{A^{+,a}}(C_n, -)$, which are also annihilated by $\pi^{1/p^n}$. Since these properties hold for all $n \geq 1$, it follows that the map $\sigma$ is an almost isomorphism, which completes the proof.
\end{proof}

With the above preparation, we now start to prove the Theorem \ref{Thm: SC over v}.
\begin{proof}[Proof of Theorem \ref{Thm: SC over v}]
By the equivalence between locally almost free sheaves and almost vector bundles (Theorem \ref{Thm: vb is svb}), we may assume, after replacing $X$ with a suitable affinoid cover, that for every $k \geq 0$, there exists a morphism of $\mathcal{O}_{X}^{+,a}$-modules
\[
h_k: (\mathcal{O}_{X}^{+,a})^{\oplus N} \to \mathcal{V}|_{X}
\]
such that the kernel and cokernel of $h_k$ are killed by $\pi^{1/p^k}$.

For every point $x\in X$, denote the complete residue field by $k(x)$. Note that there is a canonical morphism
\[
A^+\to \prod_{x\in X}\widetilde{k(x)}^+.
\]
where $A^+ = \mathcal{O}_X^+(X)$. Fix a pseudo-uniformizer $\varpi \in A^+$. One checks directly from the definition that $\prod_{x\in X}\widetilde{k(x)}^+$ is an integral perfectoid ring, since each $k(x)^+$ is. We then define
\[
U := \mathrm{Spa}\left( \left( \prod_{x \in X} \widetilde{k(x)}^+ \right) \left[ \frac{1}{\varpi} \right], \prod_{x \in X} \widetilde{k(x)}^+ \right).
\]

We first verify that $U \to X$ is a $v$-cover by showing that the map on underlying topological spaces $|U| \to |X|$ is surjective. Given any valuation $x\in X$, corresponding to a valuation point $\spa(k(x),k(x)^+)$. This gives a valuation on $k(x)$ and hence induces a valuation on $\widetilde{k(x)}$ with valuation ring $\widetilde{k(x)}^+$. By projecting to the component $\widetilde{k(x)}$, this defines a valuation on
\[
\Bigl(\Bigl(\prod_{y\in X\setminus\{x\}}\widetilde{k(y)}^+\Bigr)\oplus \widetilde{k(x)}\Bigr)\Bigl[\frac{1}{\varpi}\Bigr]
\cong \prod_{x\in X}\widetilde{k(x)}^+\Bigl[\frac{1}{\varpi}\Bigr],
\]
which yields the desired surjectivity.
Thus,
\[
\CV(U)\cong \CV(X)\widehat\otimes_{A^+}\CO^{+,a}_U(U)
\cong \CV(X)\widehat\otimes_{A^{+,a}}(\prod_{x\in X}\widetilde{k(x)}^+)^a.
\]
By the above lemma \ref{lemma of commut} and Theorem \ref{Thm: al free is free}, we obtain
\[
\CV(X)\widehat\otimes_{A^{+,a}}\prod_{x\in X}\widetilde{k(x)}^{+,a}
\cong \prod_{x\in X}\left(\CV(X)\widehat\otimes_{A^{+,a}}\widetilde{k(x)}^{+,a}\right)
\cong \left[\left(\prod_{x\in X}\widetilde{k(x)}^+\right)^a\right]^N.
\]
This shows that $\mathcal{V}|_U \cong (\mathcal{O}_U^{+,a})^{\oplus N}$, completing the proof.
\end{proof}

\section{$Arc$-vector Bundles}\label{section:prove arc=v}

The goal of this section is to prove Theorem \ref{theo: avb=arc vb}. First, we recall some basics about the $arc$-topology. Throughout this section, we fix an integral perfectoid ring \( R \) and a pseudo-uniformizer \( \varpi \).
\begin{definition}
	Define the category \( \mathrm{Perf}^{\mathrm{tf}}_{R} \) as follows:
	\begin{enumerate}
		\item The objects are all affinoid formal schemes \( \Spf(S) \), where \( S \) is an integral perfectoid algebra over \( R \), endowed with the \( \varpi \)-adic topology, and the localization \( S \to S[\frac{1}{\varpi}] \) is injective.
		\item The morphisms are all morphisms of \( R \)-formal schemes.
	\end{enumerate}
\end{definition}
It is not difficult to see that \( \mathrm{Perf}^{\mathrm{tf}}_{R} \) has all finite limits. Recall the following definition (cf. \cite[Definition 8.7]{Bhatt_2022}):
\begin{definition}[$arc$-cover]
	Let \( f: X_1 \to X_2 \) be a morphism in \( \mathrm{Perf}_{R}^{\mathrm{tf}} \). We say that \( f \) is an \( arc \)-cover if for every rank \( 1 \) \( \varpi \)-adically complete valuation ring \( V_1 \) and every morphism \( v_1: \Spf(V_1) \to X_1 \), there exists a \( \varpi \)-adically complete rank \( 1 \) valuation ring extension \( V_2 \) of \( V_1 \) and a ring homomorphism \( v_2: \Spf(V_2) \to X_2 \) such that the following diagram commutes:
	\[
	\begin{tikzcd}
	\Spf(V_1) \ar{r} \ar{d}{v_1} & \Spf(V_2) \ar{d}{v_2} \\
	X_1 \ar{r} & X_2
	\end{tikzcd}.
	\]
\end{definition}
It is not difficult to see that \( \mathrm{Perf}_{R}^{\mathrm{tf}} \), together with its \( arc \)-topology, forms a Grothendieck topology. It is important to note that the functor \( \mathbb{G}_a: X \to \Gamma(X, \mathcal{O}_X) \) is generally not a sheaf in the \( arc \)-topology. In fact, one can verify that if there exists an injection \( f: R \to R_1 \) such that \( f \) is an almost isomorphism but not an isomorphism, then the corresponding \( \Spf(R_1) \to \Spf(R) \) is an \( arc \)-cover. In this case, \( \mathbb{G}_a \) does not satisfy the sheaf axiom for this cover.

\begin{definition}[Structural sheaf of $arc$-topology]\label{def: arc sheaf}
	Define \( \mathcal{O}_{\mathrm{Perf}_{R}^{\mathrm{tf}}} \) to be the \( arc \)-sheafification of the functor \( \mathbb{G}_a \), called the structure sheaf of the \( arc \)-topology.
\end{definition}

The sheaf \( \mathcal{O}_{\mathrm{Perf}_{R}^{\mathrm{tf}}} \) admits a more direct description. For any \( X = \Spf(S) \in \mathrm{Perf}_{R}^{\mathrm{tf}} \), recall that \( S_* \) denotes the ring \( \mathrm{Hom}_{R}(R^\circ, S) \), which is isomorphic to the ring of power-bounded elements in \( S[\frac{1}{\varpi}] \).

\begin{theorem}\label{theo: structure arc= power bounded}
	For any \( X = \Spf(S) \in \mathrm{Perf}_{R}^{\mathrm{tf}} \), there is a canonical isomorphism \( \Gamma(X, \mathcal{O}_{\mathrm{Perf}_{R}^{\mathrm{tf}}}) \cong S_* \).
\end{theorem}

\begin{proof}
    First, we show that the functor $\mathbb{G}_a^{sh}: X = \Spf(S) \in \mathrm{Perf}^{\mathrm{tf}}_{R} \mapsto S_*$ is an $arc$-sheaf. This is equivalent to proving that for any homomorphism $f: S_1 \to S_2$ of integral perfectoid algebras over $R$ such that the corresponding morphism $\Spf(f): \Spf(S_2) \to \Spf(S_1)$ is an $arc$-cover, the associated \v{C}ech complex
\[ 0 \to (S_1)_* \to (S_2)_* \to (S_{2} \widehat{\otimes}_{S_1} S_2)_* \]
is exact.

For $i \in \{1, 2\}$, let $S_i'$ be the subring of $S_i[\frac{1}{\varpi}]$ generated by $R$ and $S_i^{\circ\circ}$, and let $(S_{i}')^+$ be the integral closure of $S_{i}'$ in $S_{i}[\frac{1}{\varpi}]$. Then the localization $f[\frac{1}{\varpi}]: S_1[\frac{1}{\varpi}] \to S_2[\frac{1}{\varpi}]$ sends $(S_1')^+$ into $(S_2')^+$. We will show that the induced map of perfectoid spaces $\Spa(f)': \Spa(S_2[\frac{1}{\varpi}], (S_2')^+) \to \Spa(S_1[\frac{1}{\varpi}], (S_1')^+)$ is a $v$-cover. If this is established, then by \cite[Theorem 17.1.3]{Scholze_2020}, the \v{C}ech complex induced by $f$,
\begin{equation}\label{eq: arc structure}0 \to S_1\left[\frac{1}{\varpi}\right] \to S_2\left[\frac{1}{\varpi}\right] \to S_2 \widehat{\otimes}_{S_1} S_2\left[\frac{1}{\varpi}\right]\end{equation}
is exact. Applying the functor $(-)^\circ$ to each term of the exact sequence (\ref{eq: arc structure}) yields the required exact sequence.

We now verify that $\Spa(f)': \Spa(S_2[\frac{1}{\varpi}], (S_2')^+) \to \Spa(S_1[\frac{1}{\varpi}], (S_1')^+)$ is a $v$-cover. By definition, it suffices to prove that for any perfectoid Tate–Huber pair $(K, K^+)$ where $K$ is a perfectoid field and $K^+$ is a valuation ring, and for any morphism
\[ x: (S_1, (S_1')^+) \to (K, K^+), \]
there exists a valuation extension $i: (K, K^+) \to (L, L^+)$ and a morphism $y: (S_2, S_2^+) \to (L, L^+)$ such that the following diagram commutes:
\begin{equation}\label{eq: arc-v diagram}
\begin{tikzcd}
    S_1 \ar{r}{f[\frac{1}{\varpi}]} \ar{d}{x} & S_2 \ar{d}{y} \\
    K \ar{r}{i} & L.
\end{tikzcd}
\end{equation}
Let $K^\circ$ be the subring of power-bounded elements in $K$. Then $K^\circ$ is a rank $1$ valuation ring containing $K^+$. Since $\Spf(f)$ is an $arc$-cover, there exists a rank $1$ valuation field extension $i: K \to L$ and a continuous homomorphism $y: S_2[\frac{1}{\varpi}] \to L$ such that diagram (\ref{eq: arc-v diagram}) commutes. Let $L^+ \subseteq L^\circ$ be a valuation extension of $K^+$. By construction, $y(S_2') \subset L^+$. Consequently, as $(S_2')^+$ is the integral closure of $S_2'$ in $S_2[\frac{1}{\varpi}]$, we have $y\big((S_2')^+\big) \subseteq L^+$. This completes the proof.

Finally, we verify that the functor $\mathbb{G}_a^{sh}: X = \Spf(S) \in \mathrm{Perf}^{\mathrm{tf}}_{R} \mapsto S_*$ is the $arc$-sheafification of $\mathbb{G}_a$. By definition, there is a morphism $\mathbb{G}_a \to \mathbb{G}_a^{sh}$. Note that for any $X = \Spf(S) \in \mathrm{Perf}^{\mathrm{tf}}_{R}$, letting $X' = \Spf(S_*)$, the $arc$-cover $X' \to X$ satisfies that the canonical map from $\mathbb{G}_a(X')$ to $\mathbb{G}_a^{sh}(X')$ is an isomorphism. This proves that $\mathbb{G}_a^{sh}$ is the sheafification of $\mathbb{G}_a$.
\end{proof}

For \( S \in \mathrm{Perf}^{\mathrm{tf}}_{R} \), let \( S^+ \) be the integral closure of \( S \) in \( S[\frac{1}{\varpi}] \). Then \( (S[\frac{1}{\varpi}], S^+) \) forms a perfectoid Tate–Huber pair. Let \( U = \Spa(R[\frac{1}{\varpi}], R^+) \). The previous construction gives a functor \( \iota: \mathrm{Perf}^{\mathrm{tf}}_{R} \to U_v \).

\begin{definition}
	Let \( \mathbb{G}_a^\circ \) be the functor on \( U_v \) defined by \( \Spa(S, S^+) \mapsto S^\circ \).
\end{definition}

Then, again by \cite[Theorem 17.1.3]{Scholze_2020}, \( \mathbb{G}_a^{\circ} \) is a sheaf.

\begin{proposition}
    There is an equivalence of categories between the category of locally free \( \mathcal{O}_{\mathrm{Perf}^{\mathrm{tf}}_{R}} \)-modules on \( \mathrm{Perf}^{\mathrm{tf}}_{R} \) and the category of locally free \( \mathbb{G}_a^\circ \)-modules on \( U_v \).
\end{proposition}

\begin{proof}
	We first construct two functors. For a locally free \( \mathcal{O}_{\mathrm{Perf}^{\mathrm{tf}}_{R}} \)-module \( \mathcal{M} \) on \( \mathrm{Perf}^{\mathrm{tf}}_{R} \), define a functor \( \mathcal{M}_{v} \) that sends any \( \Spa(S, S^+) \in U_v \) to \( \Gamma(\Spf(S^\circ), \mathcal{M}) \). For a locally free \( \mathbb{G}_a^\circ \)-module \( \mathcal{N} \) on \( U_v \), define \( \mathcal{N}_{arc} \) that sends any \( S \in \mathrm{Perf}^{\mathrm{tf}}_{R} \) to \( \Gamma\big(\Spa(S[\frac{1}{\varpi}], S^+), \mathcal{N}\big) \).

	Next, we verify that the constructed \( \mathcal{M}_{v} \) is a locally free \( \mathbb{G}_a^{\circ} \)-module on \( U_v \). Suppose \( X = \Spf(S) \in \mathrm{Perf}^{\mathrm{tf}}_{R} \) is such that \( \mathcal{M}|_{X} \cong \mathcal{O}_{\mathrm{Perf}^{\mathrm{tf}}_{R}}^{\oplus r} \). Let \( S' \) be the subring of \( S[\frac{1}{\varpi}] \) generated by \( S^{\circ\circ} \), and let \( (S')^+ \) be the integral closure of \( S' \) in \( S[\frac{1}{\varpi}] \). Then, by the proof of Theorem \ref{theo: structure arc= power bounded}, \( V := \Spa(S[\frac{1}{\varpi}], S') \) is a \( v \)-cover of \( U \) and \( \mathcal{M}_v|_{V} \cong (\mathbb{G}_a^{\circ})^{\oplus r} \).

	We then verify that the constructed \( \mathcal{N}_{arc} \) is a locally free \( \mathcal{O}_{\mathrm{Perf}_{R}^{\mathrm{tf}}} \)-module on \( \mathrm{Perf}_{R}^{\mathrm{tf}} \). Suppose \( V := \Spa(S, S^+) \in U_v \) is a \( v \)-cover of \( U \) such that \( \mathcal{N}_{arc}|_{V} \cong (\mathbb{G}_a^{\circ})^{\oplus r} \). Let \( X := \Spf(S^+) \in \mathrm{Perf}_{R}^{\mathrm{tf}} \). Then \( X \) is an \( arc \)-cover of \( \Spf(R) \) and \( \mathcal{N}_{arc}|_{X} \cong (\mathcal{O}_{\mathrm{Perf}_{R}^{\mathrm{tf}}})^{\oplus r} \).

	It remains to verify that \( (\mathcal{M}_{v})_{arc} \cong \mathcal{M} \) and \( (\mathcal{N}_{arc})_v \cong \mathcal{N} \). By definition, $(\mathcal{M}_{v})_{arc}(S)=\mathcal{M}(S^[\frac{1}{\varpi}]^\circ)$, which is canonically isomorphic to $\mathcal{M}(S)$ by the proof of Theorem \ref{theo: structure arc= power bounded}. Also, \[(\mathcal{N}_{arc})_{v}\big(\Spa(S,S^+)\big)=\mathcal{N}\big(\Spa(S,S^\circ)\big)\] for any $\Spa(S,S^+)\in U_v$. Hence, we only need to prove that
    \[\Gamma(\Spa(S,S^+),\mathcal N)\cong \Gamma(\Spa(S,S^\circ),\mathcal{N})\]
    for any $\Spa(S,S^+)\in U_v$. Choose a $v$-cover $\Spa(T,T^+)\to \Spa(S,S^+)$ so that $\mathcal{N}|_{\Spa(T,T^+)}\cong (\mathbb{G}_a^\circ)^{\oplus r}$ and put
    \[\Spa(T,T^{+}_1)=\Spa(T,T^+)\times_{\Spa(S,S^+)}\Spa(S,S^\circ).\]
    Then, we have the following diagram
    \[\begin{tikzcd}
        0\ar{r} & \Gamma(\Spa(S,S^+),\mathcal{N})\ar{r} & \Gamma(\Spa(T,T^+),\mathcal{N})\ar{r} & \Gamma(\Spa(T,T^+)\times_{\Spa(S,S^+)}\Spa(T,T^+),\mathcal{N})\\
        0\ar{r} & \Gamma(\Spa(S,S^\circ),\mathcal{N})\ar{r}\ar{u} & \Gamma(\Spa(T,T_1^+),\mathcal{N})\ar{u}\ar{r} & \Gamma(\Spa(T,T^+_1)\times_{\Spa(S,S^\circ)}\Spa(T,T^+_1),\mathcal{N})\ar{u}
    \end{tikzcd}\]
    with all rows exact. To show the exactness of the first column, it is enough to show the exactness of the other two columns. This is obvious since $\mathcal{N}$ is isomorphic to $(\mathbb{G}_a^{\circ})^{\oplus r}$ on $\Spa(T,T^+)$.
\end{proof}

\begin{proof}[Proof of Theorem \ref{theo: avb=arc vb}]
    By Theorem \ref{Thm: SC over v} and \ref{SAAVB}, it suffices to prove that the category of $\mathcal{O}^{+,a}$-locally free sheaves on $U_v$ is equivalent to the category of $\mathbb{G}_a^\circ$-locally free sheaves on $U_v$.

    For an $\mathcal O^{+,a}$-locally free sheaf $\mathcal{M}$ on $U_v$, define $\mathcal{M}_*$ as the functor \[\mathcal{M}_*(V):=\Gamma(V,\mathcal{M})_*=\mathrm{Hom}_{R}(R^{\circ\circ},\Gamma(V,\mathcal{M})). \] By the left exactness of $(-)_*$, $\mathcal{M}_*$ is a sheaf. Then, since for any $S\in \mathrm{Perf}_{R}^{\mathrm{tf}}$, $S_*=S[\frac{1}{\varpi}]^\circ$, $\mathcal{M}_*$ is a $\mathbb{G}_a^\circ$-locally free sheaf on $U_v$. Conversely, for any $\mathbb{G}_a^\circ$-locally free sheaf $\mathcal{N}$ on $U_v$, define $\mathcal{N}^a$ as its almostification. It is not difficult to see that $\mathcal N^a$ is an $\mathcal{O}^{+,a}$-locally free sheaf. Using Theorem \ref{theo: structure arc= power bounded} on all $V\in U_v$, we obtain that the functor $(\mathcal{M}\mapsto \mathcal{M}_*)$ and the functor $(\mathcal{N}\mapsto \mathcal{N}^a)$ are quasi-inverse to each other.
\end{proof}

\bibliographystyle{alpha}
\bibliography{ref}
\addcontentsline{toc}{section}{References}
\end{document}